\documentclass{amsproc}
\usepackage[all]{xy}
\usepackage{amssymb}

\usepackage{color}
\newtheorem{theorem}{Theorem}[section]
\newtheorem{lemma}[theorem]{Lemma}
\newtheorem{corollary}[theorem]{Corollary}
\newtheorem{proposition}[theorem]{Proposition}

\theoremstyle{definition}
\newtheorem{definition}[theorem]{Definition}
\newtheorem{example}[theorem]{Example}

\theoremstyle{remark}
\newtheorem{remark}[theorem]{Remark}

\numberwithin{equation}{section}

\DeclareMathOperator{\Hom}{Hom}
\DeclareMathOperator{\Ext}{Ext} 
\DeclareMathOperator{\End}{End}

\DeclareMathOperator{\Tor}{Tor} 

\DeclareRobustCommand\longtwoheadrightarrow
     {\relbar\joinrel\twoheadrightarrow}

\newcommand{\Dcal}{\ensuremath{\mathcal{D}}}
\newcommand{\Xcal}{\ensuremath{\mathcal{X}}}
\newcommand{\Ycal}{\ensuremath{\mathcal{Y}}}
\newcommand{\Tcal}{\ensuremath{\mathcal{T}}}
\newcommand{\Fcal}{\ensuremath{\mathcal{F}}}
\newcommand{\Ccal}{\ensuremath{\mathcal{C}}}

\newcommand{\Wcal}{\ensuremath{\mathcal{W}}}

\newcommand{\tube}{\ensuremath{\mathbf{t}}}

\newcommand{\Add}[1]{\mbox{\rm{Add}}(#1)}

\newcommand{\add}[1]{\mbox{\rm{add}}(#1)}

\newcommand{\Gen}[1]{\mbox{\rm{Gen}}(#1)}
\newcommand{\gen}[1]{\mbox{\rm{gen}}(#1)}

\newcommand{\Filt}[1]{\mbox{\rm{Filt}}(#1)}
\newcommand{\FiltGen}[1]{\mbox{\rm{FiltGen}}(#1)}

\newcommand{\ModA}{\ensuremath\mbox{\rm{Mod}-$A$}}
\newcommand{\modA}{\ensuremath\mbox{\rm{mod}-$A$}}
\newcommand{\ModB}{\ensuremath\mbox{\rm{Mod}-$B$}}

\newcommand{\Ker}[1]{\mbox{\rm{Ker}}(#1)}
\newcommand{\Coker}[1]{\mbox{\rm{Coker}}(#1)}
\newcommand{\Img}[1]{\mbox{\rm{Im}}(#1)}

\newcommand{\ProjA}{\ensuremath\mbox{\rm{Proj}-$A$}}
\newcommand{\projA}{\ensuremath\mbox{\rm{proj}-$A$}}

\newcommand{\ra}{\longrightarrow}

\begin{document}

\title{A characterisation of $\tau$-tilting finite algebras}

\author{Lidia Angeleri H\"ugel}
\address{Lidia Angeleri H\"ugel, Dipartimento di Informatica - Settore di Matematica, Universit\`a degli Studi di Verona, Strada le Grazie 15 - Ca' Vignal, I-37134 Verona, Italy}
\email{lidia.angeleri@univr.it}
\thanks{The authors acknowledge funding from the Project ``Ricerca di Base 2015'' of the University of Verona. Additionally, the third named author acknowledges support from the Department of Computer Sciences of the University of Verona in the earlier part of this project, as well as from the Engineering and Physical Sciences Research Council of the United Kingdom, grant number EP/N016505/1, in the later part of this project.}

\author{Frederik Marks}
\address{Frederik Marks, Institut f\"ur Algebra und Zahlentheorie, Universit\"at Stuttgart, Pfaffenwaldring 57, 70569 Stuttgart, Germany}
\email{marks@mathematik.uni-stuttgart.de}

\author{Jorge Vit\'oria}
\address{Jorge Vit\'oria, Department of Mathematics, City, University of London, Northampton Square, London EC1V 0HB, United Kingdom}
\email{jorge.vitoria@city.ac.uk}




\begin{abstract}
We prove that a finite dimensional algebra is $\tau$-tilting finite if and only if it does not admit large silting modules. Moreover, we show that for a $\tau$-tilting finite algebra $A$ there is a bijection between isomorphism classes of basic support $\tau$-tilting (that is, finite dimensional silting) modules and equivalence classes of ring epimorphisms $A\longrightarrow B$ with $\Tor_1^A(B,B)=0$. It follows that a finite dimensional algebra is $\tau$-tilting finite if and only if there are only finitely many equivalence classes of such ring epimorphisms. 
\end{abstract}

\maketitle


\section*{Introduction}
Adachi, Iyama and Reiten  \cite{AIR} defined  support $\tau$-tilting modules  in the context of finite dimensional algebras,  motivated by the study of mutation in cluster theory. Their work has inspired the notion of a silting module, which was introduced  in  \cite{AMV1} with the aim of exporting support  $\tau$-tilting modules to arbitrary rings. In this broader context, however,  it is necessary to allow for large modules. For example, over a commutative ring $A$  there are no non-trivial finitely presented silting modules, but there is a rich supply of large silting modules  parametrising the Gabriel localisations of $A$, see \cite{AH} for details. Also over finite dimensional algebras some phenomena only become visible  when leaving the finite dimensional world.  For example, as illustrated in \cite{AMV2}, certain posets associated to the algebra, such as the poset of noncrossing partitions considered in \cite{IT}, can be  completed to lattices by dropping finiteness assumptions. 
Hence, it is often useful to pass to large modules - provided they do exist.
 
A well-known result due to Auslander states that a finite dimensional algebra does not admit  large indecomposable modules if and only if it has finite representation type. 
In this note, we will ask about the existence of large  silting modules,  that is, infinite dimensional silting modules which are not obtained trivially as direct sums  of finite dimensional ones. We will  prove that a finite dimensional algebra  does not admit large silting modules if and only if it is $\tau$-tilting finite. This  class of algebras introduced in \cite{DIJ} encompasses representation-finite algebras, local  algebras, and preprojective algebras of Dynkin type \cite{Mi}. An algebra $A$  is called $\tau$-tilting finite if there are only finitely many isomorphism classes of basic support $\tau$-tilting (that is, finite dimensional silting) modules. We show that the latter correspond bijectively to equivalence classes of ring epimorphisms $A\longrightarrow B$ with $\Tor_1^A(B,B)=0$. As a consequence, we obtain that a finite dimensional algebra $A$ is $\tau$-tilting finite if and only if there are only finitely many equivalence classes of such ring epimorphisms. 
 

\subsection*{Notation}
Throughout this note, let $A$ be a (unital) ring, $\ModA$ the category of right $A$-modules, and $\modA$ its subcategory of finitely presented modules. We denote by $\ProjA$ and $\projA$ the class of all projective and of all finitely generated projective right $A$-modules, respectively. Furthermore, we write $\mathsf{D}(A)$ for the derived category of $\ModA$. Given a module $M\in\ModA$,  we denote by $\Add{M}$ the class of all modules which are isomorphic to direct summands of direct sums of copies of $M$, by $\Gen{M}$ the class of all $M$-generated modules, i.~e.~all modules $N$ admitting a surjection $M^{(I)}\ra N$ for some set $I$, and we set $\gen{M}=\Gen{M}\cap\modA$. Moreover, we define
$$M^\circ=\{X\in \ModA \,\mid\, \Hom_A(M,X)=0\}$$
$$M^{\perp_1}=\{X\in \ModA \,\mid\, \Ext_A^1(M,X)=0\}.$$ 
All finite dimensional algebras considered in this note are finite dimensional over an algebraically closed field $\mathbb{K}$.


\section{Silting modules and ring epimorphisms}
\subsection{Silting modules}
We first introduce silting modules over a ring $A$.
Given a set $\Sigma$  of morphisms in  $\ProjA$, we define the subcategory
$$\Dcal_\Sigma:=\{X\in \ModA \,\mid\, \Hom_A(\sigma,X)\ \text{is surjective for all } \sigma\in\Sigma \}.$$
Note that, if we identify a morphism $\sigma$ with its cone in $\mathsf{D}(A)$, then we can write
$$\Dcal_\Sigma=\{X\in \ModA \,\mid\, \Hom_{\mathsf{D}(A)}(\sigma,X[1])=0\ \text{for all } \sigma\in\Sigma \}.$$
If $\Sigma$ contains just one element $\sigma$, we write $\Dcal_\sigma$.

\begin{definition}\cite{AMV1}\label{def p silting}
We say that an $A$-module $T$ is
\begin{itemize}
\item  \emph{partial silting} if it admits a projective presentation $ P\stackrel{\sigma}{\longrightarrow} Q\ra T\ra 0$ such that 
$\mathcal{D}_\sigma$ is a torsion class containing $T$;
\item \emph{silting} if it admits  a projective presentation $ P\stackrel{\sigma}{\longrightarrow} Q\ra T\ra 0$ such that 
$\mathcal{D}_\sigma=\Gen{T}$.
\end{itemize}
We say that $T$ is (partial) silting \emph{with respect to} $\sigma$. The torsion class $\Gen{T}$ generated by a silting module $T$ is called a \emph{silting class}. Two silting modules $T$ and $T'$ are said to be \emph{equivalent} if they generate the same silting class or, equivalently, $\Add{T}=\Add{T^\prime}$. 
\end{definition}
 
By  \cite[Theorem 3.12]{AMV1}  every partial silting module $T_1$ with respect to a presentation $\sigma_1$ can be completed to a silting module ${T}$, called the {\em Bongartz completion} of $T_1$, such that $\Gen{T}=\Dcal_{\sigma_1}$. It was also shown in \cite{AMV1} that every silting module $T$ induces an exact sequence of the form
\begin{equation}\label{approx seq} 
\xymatrix{A\ar[r]^f&T_0\ar[r]&T_1\ar[r]&0}
\end{equation}
where $T_0$ and $T_1$ are in $\Add{T}$ and $f$ is a left $\Add{T}$-approximation.
 
Recall that a module $T$ is \emph{tilting} if it is silting with respect to a projective resolution $ 0\longrightarrow P\stackrel{\sigma}{\longrightarrow} Q\ra T\ra 0$, or equivalently, if $T^{\perp_1}=\Gen{T}$. In this case, the module $T_1$ in (\ref{approx seq}) is a partial silting module (actually, partial tilting).
We do not know, however, whether this is always true. In this subsection we explore some cases that allow us to conclude that $T_1$ is partial silting. 

We say that an $A$-module $M$ admits a \emph{presilting presentation} if there are $P$ and $Q$ in $\ProjA$ and an exact sequence
$$\xymatrix{P\ar[r]^{\sigma} & Q\ar[r] & M\ar[r] & 0}$$
such that $M^{(I)}\in\mathcal{D}_{\sigma}$ or, equivalently, 
$\Hom_{\mathsf{D}(A)}({\sigma},{{\sigma}^{(I)}[1]})=0$, for all sets $I$.

\begin{example}\label{prespres}
\begin{enumerate}
\item It follows from the definition above that every partial silting module admits a presilting presentation.
\item If a module $M$ admits both a minimal projective presentation $p$ and a presilting presentation, then $p$ is a presilting presentation. Indeed, $p$ must be a summand of the presilting presentation, making it presilting as well.
\end{enumerate}
\end{example}

\begin{proposition}\label{Prop silting}
Let $T_0$ be an $A$-module, and let $A\overset{f}{\ra} T_0\ra T_1\ra 0$ be an $\Add{T_0}$-approximation sequence. If $T_0$ admits a presilting presentation, then $T=T_0\oplus T_1$ is a silting $A$-module, and $T_1$ is a partial silting module with respect to a projective presentation $\sigma_1$ such that $\Gen{T}=\Dcal_{\sigma_1}$.
\end{proposition}

\begin{proof}
Let $\sigma_0$ be a presilting presentation of $T_0$. Then, the $A$-module map $f$ can be lifted to a morphism of chain complexes $A\longrightarrow \sigma_0$. Consider the induced triangle in the derived category $\mathsf{D}(A)$
$$\xymatrix{A\ar[r] & \sigma_0\ar[r] & \sigma_1\ar[r] & A[1].}$$
By construction, $\sigma_1$ is a projective presentation of $T_1$. We show that $T=T_0\oplus T_1$ is a silting $A$-module with respect to the presentation $\sigma_0\oplus\sigma_1$, that is, $\Gen{T}=\Dcal_{\sigma_0\oplus\sigma_1}$. First, observe that $\Dcal_{\sigma_0\oplus\sigma_1}=\Dcal_{\sigma_1}$. The inclusion $\Dcal_{\sigma_0\oplus\sigma_1}\subseteq\Dcal_{\sigma_1}$ is clear. For the converse, apply the functor $\Hom_{\mathsf{D}(A)}({-},{X[1]})$ for $X$ in $\Dcal_{\sigma_1}$ to the triangle above. Since $\Hom_{\mathsf{D}(A)}({A},{X[1]})=0=\Hom_{\mathsf{D}(A)}({\sigma_1},{X[1]})$, we also get $\Hom_{\mathsf{D}(A)}({\sigma_0},{X[1]})=0$ or, in other words, $X\in\Dcal_{\sigma_0}$. Hence, we conclude that $X$ lies in $\Dcal_{\sigma_0}\cap\Dcal_{\sigma_1}=\Dcal_{\sigma_0\oplus\sigma_1}$.

It remains to show that $\Gen{T}=\Gen{T_0}=\Dcal_{\sigma_1}$. Using that $f$ is a left $\Add{T_0}$-approximation, we see that an $A$-module $X$ belongs to $\Gen{T_0}$ if and only if the map $\Hom_{A}({f},{X}):\Hom_{A}({T_0},{X})\longrightarrow \Hom_{A}({A},{X})$  is surjective. This amounts to the surjectivity of the map $\pi:\Hom_{\mathsf{D}(A)}({\sigma_0},{X})\longrightarrow \Hom_{\mathsf{D}(A)}({A},{X})$ which is obtained by applying the functor $\Hom_{\mathsf{D}(A)}({-},{X})$ to the triangle above. 
Considering the  exact sequence 
$$\Hom_{\mathsf{D}(A)}({\sigma_0},{X})\stackrel{\pi}{\longrightarrow}\Hom_{\mathsf{D}(A)}({A},{X})\rightarrow\Hom_{\mathsf{D}(A)}({\sigma_1},{X[1]})\rightarrow\Hom_{\mathsf{D}(A)}({\sigma_0},{X[1]}),$$
it follows immediately that $\Gen{T_0}\supset\Dcal_{\sigma_1}$. Moreover,  since
$\sigma_0$ is a presilting presentation of $T_0$, the class $\Dcal_{\sigma_0}$ contains $\Add{T_0}$ and therefore also $\Gen{T_0}$. So, if $X\in\Gen{T_0}$, then $\Hom_{\mathsf{D}(A)}({\sigma_0},{X}[1])=0$ and $\pi$ is surjective, hence $\Hom_{\mathsf{D}(A)}({\sigma_1},{X[1]})=0$. This proves that $\Gen{T_0}\subset\Dcal_{\sigma_1}$. 
\end{proof}

\begin{corollary}\label{perfect} 
If $T$ is a silting module such that every module in $\Add{T}$ admits a minimal projective presentation, then any module $T_1$ appearing in a left $\Add{T}$-approximation sequence of the form (\ref{approx seq}) is partial silting with respect to a projective presentation $\sigma_1$ such that $\Dcal_{\sigma_1}=\Gen{T}$.
\end{corollary}
\begin{proof}
By Proposition~\ref{Prop silting}, it is enough to show that $T_0$ admits a presilting presentation. Since $T_0$ lies in $\Add{T}$, it is a summand of $T^{(I)}$, for some set $I$. Hence, the minimal projective presentation of $T_0$ is a summand of the minimal projective presentation of $T^{(I)}$. Since $T$ has a presilting presentation, it follows from Example~\ref{prespres}(2) that the minimal projective presentation of $T^{(I)}$ is presilting and, therefore, so is the minimal projective presentation of $T_0$.
\end{proof}

There is a different way of interpreting the statements above.
Recall from \cite{Wei} that if $T$ is a silting module with respect to $\sigma$, then there is a triangle
\begin{equation}\label{approx tria}
\xymatrix{A\ar[r]^\phi&\sigma_0\ar[r]&\sigma_1\ar[r]&A[1]}
\end{equation}
in $\mathsf{D}(A)$, where $\sigma_0$ and $\sigma_1$ lie in $\Add{\sigma}$ and $\phi$ is a left $\Add{\sigma}$-approximation of $A$. Note that applying the functor $H^0(-)$ to the triangle above yields a left $\Add{T}$-approximation sequence in $\ModA$ as in (\ref{approx seq}) for which $T_1:=H^0(\sigma_1)$ is partial silting, by Proposition \ref{Prop silting}. In fact, if every module in $\Add{T}$ admits a minimal projective presentation, then every exact sequence as in (\ref{approx seq}) can be lifted to a triangle as in $(\ref{approx tria})$. This is precisely Corollary \ref{perfect}. 

\begin{definition} 
Let $T$ be a silting module with respect to $\sigma$. If the map $\phi$ in the triangle (\ref{approx tria}) above can be chosen left-minimal (i.~e.~any $h\in\End_{\mathsf{D}(A)}(\sigma_0)$ with $hf=f$ is an isomorphism), then $T$ is said to be a {\em minimal} silting module. 
\end{definition} 

Using that left-minimal maps are essentially unique, we can associate to every minimal silting module $T$ a canonical partial silting module $T_1:=H^0(\mathsf{cone}(\phi))$. This assignment will become useful later in this note.

\begin{remark} If $A$ is (right) hereditary or perfect, then a silting module $T$ with respect to $\sigma$ is minimal if and only if there is an exact sequence as in (\ref{approx seq}) for which the left $\Add{T}$-approximation $f$ can be chosen left-minimal, compare with \cite[Definition 5.4]{AMV2}.

Indeed, it is true for any ring $A$ that a left-minimal $\Add{\sigma}$-approximation of $A$ in $\mathsf{D}(A)$ induces a left-minimal $\Add{T}$-approximation of $A$ in $\ModA$. To see this, given a left-minimal $\Add{\sigma}$-approximation $\phi:A\longrightarrow \sigma_0$, we check that for any endomorphism $g$ of $H^0(\sigma_0)$ fulfilling $gH^0(\phi)=H^0(\phi)$ there is an endomorphism $\tilde{g}$ of $\sigma_0$ in $\mathsf{D}(A)$ such that $H^0(\tilde{g})=g$ and $\tilde{g}\phi=\phi$. The minimality of $\phi$ then guarantees that $\tilde{g}$ is an isomorphism and, hence, so is $g$.

For the converse, we first need to assure that a left $\Add{T}$-approximation $f:A\longrightarrow T_0$ in $\ModA$ can be lifted to a left $\Add{\sigma}$-approximation $\phi:A\longrightarrow\sigma_0$ in $\mathsf{D}(A)$. Here, we use that if $A$ is (right) hereditary or perfect, then $H^0$ induces a dense functor $\Add{\sigma}\longrightarrow\Add{T}$ (see also Corollary \ref{perfect}). Secondly, observe that we can choose $\phi$ to be left-minimal provided that $f$ is left-minimal. Indeed, in the hereditary (respectively, perfect) case we can choose $\phi$ with $\sigma_0$ being a monomorphic projective presentation (respectively, a minimal projective presentation) of $T_0$.
\end{remark}


\subsection{Ring epimorphisms.} We first recall some notions and basic results.
\begin{definition}
A {ring homomorphism} $f:A\longrightarrow B$  is a 
{\em ring epimorphism} if it is an epimorphism in the category of rings with unit, or equivalently, if the functor given by restriction of scalars $f_\ast:\ModB\longrightarrow \ModA$ is fully faithful.
\end{definition}

Two ring epimorphisms $f_1:A\longrightarrow B_1$ and  $f_2:A\longrightarrow B_2$ are said to be  {\em equivalent} if there is an isomorphism of rings $h: B_1\longrightarrow B_2$ such that $f_2=h\circ f_1$. We then say that $f_1$ and $f_2$ lie in the same {\em epiclass} of $A$.

\begin{definition}
A full subcategory $\Xcal$ of $\ModA$ is called {\em bireflective} if the inclusion functor $\Xcal\longrightarrow\ModA$ admits both a left and right adjoint, or equivalently, if it is closed under products, coproducts, kernels and cokernels. A full subcategory $\Wcal$ of $\modA$ is said to be {\em wide} if it is closed under kernels, cokernels and extensions.
\end{definition}

\begin{theorem}\label{epicl}
\begin{enumerate}
\item \cite{GdP,GL} There is a bijection between
\begin{itemize}
\item epiclasses of ring epimorphisms $A\ra B$;
\item bireflective subcategories of $\ModA$
\end{itemize}
mapping a ring epimorphism  $A\ra B$ to  the essential image $\Xcal_B$ of its restriction functor.
\item \cite[Theorem 4.8]{Sch} The bijection in (1) restricts to a bijection between  
\begin{itemize}
\item epiclasses of ring epimorphisms $A\ra B$ with $\Tor_1^A(B,B)=0$;
\item bireflective subcategories closed under extensions in $\ModA$.
\end{itemize}
\item \cite[Theorem 1.6.1]{I} If $A$ is a finite dimensional algebra, mapping a ring epimorphism $A\ra B$ to $\Xcal_B\cap\modA$ defines a bijection between
\begin{itemize}
\item epiclasses of ring epimorphisms $A\ra B$ with $B$  finite dimensional  and $\Tor_1^A(B,B)=0$; 
\item functorially finite wide subcategories of $\modA$. 
\end{itemize}
\end{enumerate}
\end{theorem}

We will be particularly interested in ring epimorphisms that arise from partial silting modules.
Given a set $\Sigma$ of maps in $\ProjA$, we define the subcategory
$$\Xcal_\Sigma:=\{X\in \ModA \,\mid\, \Hom_A(\sigma,X)\ \text{is bijective for all } \sigma\in\Sigma \}.$$
Similar to before, identifying a morphism $\sigma$ with its cone in $\mathsf{D}(A)$, we can write
$$\Xcal_\Sigma=\{X\in \ModA \,\mid\, \Hom_{\mathsf{D}(A)}(\sigma,X[1])=0=\Hom_{\mathsf{D}(A)}(\sigma,X)\text{ for all }\sigma\in\Sigma \}.$$
If $\Sigma$ contains just one element $\sigma$, we write $\Xcal_\sigma$.

If $T_1$ is a partial silting module with respect to a projective presentation $\sigma_1$, then we have the inclusions $\Gen{T_1}\subseteq \Dcal_{\sigma_1}\subseteq T_1\,^{\perp_1}$, and $(\Gen{T_1}, T_1^\circ)$ is a torsion pair.  Let $T$ be the Bongartz completion of $T_1$ with $\Gen{T}=\Dcal_{\sigma_1}$. The subcategory  
$$\Xcal_{\sigma_1}=\Dcal_{\sigma_1}\cap\Coker{\sigma_1}^\circ=\Gen{T}\cap T_1^\circ$$ 
is bireflective and extension-closed by \cite[Proposition 3.3]{AMV2}. Thus,  it can  be realised as $\Xcal_B$ for a ring epimorphism $A\ra B$ as in Theorem~\ref{epicl}(2), and by \cite[Theorem 3.5]{AMV2} the ring $B$ is isomorphic to an idempotent  quotient of  the endomorphism ring of $T$.

\begin{definition} 
A ring epimorphism  arising from a partial silting module as above will be said to be a {\em silting ring epimorphism}.
\end{definition}

The following result provides a large supply of silting ring epimorphisms.

\begin{theorem}  \label{def:universallocalisation}
Let $\Sigma$ be a set of morphisms in $\projA$.
\begin{enumerate}
\item \cite[Theorem~4.1]{Sch} There is a ring homomorphism $f_\Sigma: A\longrightarrow A_\Sigma$, called the
\emph{universal localisation} of $A$ at $\Sigma$, such that
\begin{enumerate}
\item[(i)] $f_\Sigma$ is \emph{$\Sigma$-inverting,} i.e.~$\sigma\otimes_A A_\Sigma$ is an isomorphism  for every
$\sigma$ in  $\Sigma$.
\item[(ii)] $f_\Sigma$ is \emph{universal
$\Sigma$-inverting}, i.e.~every $\Sigma$-inverting ring homomorphism $f: A\longrightarrow B$
factors  uniquely through  $f_\Sigma$.
\end{enumerate}
\item \cite[Theorem 6.7]{MS2} The ring homomorphism $f_\Sigma\colon A\longrightarrow A_\Sigma$ is a silting ring epimorphism.
\end{enumerate} 
\end{theorem}
Note that $\Xcal_\Sigma$ is precisely the bireflective subcategory corresponding to $f_\Sigma$ under the bijection in Theorem~\ref{epicl}, see \cite[p.~52 and 57]{Sch}. 


\section{From torsion classes to abelian subcategories}
For any torsion class $\Tcal$ in $\ModA$, or more generally, in an abelian category $\mathcal{A}$, we  consider the following subcategory   of $\mathcal{A}$  studied in  \cite{IT} 
\begin{equation}\label{alpha}\mathfrak a(\Tcal):=\{X\in\Tcal: \text{ if } (g:Y\longrightarrow X)\in\Tcal, \text{ then } \Ker{g}\in\Tcal\}.\end{equation}

\begin{proposition}\cite[Proposition 2.12]{IT}\label{abelian} The subcategory $\mathfrak a(\Tcal)$ of $\mathcal{A}$ is closed under kernels, cokernels, and extensions.
\end{proposition}

In general, however, $\mathfrak a(\Tcal)$ is not a bireflective subcategory, not even in case $\Tcal$ is a tilting class, as the following example shows.

\begin{example}  
Let $A$ be the Kronecker algebra. It is well-known that the class $\mathcal T$ of all $A$-modules without indecomposable preprojective summands is a tilting class, that is, $\Tcal=\Gen{L}$ for a tilting module  $L$, called the Lukas tilting module. We will show that $\mathfrak{a}(\Tcal)$ is not closed under coproducts.

We first claim that  $\mathfrak{a}(\Tcal)$ contains all regular modules. By Proposition~\ref{abelian}, it is enough to check that any  simple regular module $S$ lies in $\mathfrak{a}(\Tcal)$.  Moreover, since $A$ is hereditary, we know from \cite[Proposition 2.15]{IT} that 
$$\mathfrak a(\Tcal)=\{X\in\Tcal: \text{ if } (g:Y\longtwoheadrightarrow X)\in\Tcal, \text{ then } \Ker{g}\in\Tcal\}.$$ 
Certainly, $S\in \Tcal=\Gen{L}$. Assume that there is a surjection $g: X\ra  S$  with  $X$ in $\Gen{L}$ and $K:=\Ker{g}$ not lying in $\Gen{L}$. Since $\Gen{L}$ is the class of all modules without indecomposable preprojective  summands,  $K$ must have an indecomposable preprojective summand $P$. Consider the pushout diagram
 \[\xymatrix{0\ar[r] &K \ar[d]^{\pi}\ar[r]^{}&X\ar[d]^{\pi'}\ar[r]^{g}&S\ar@{=}[d]\ar[r]{}&0\\0\ar[r] &P \ar[r]^{}&X'\ar[r]^{g'}&S\ar[r]{}&0}\]
with $\pi$ the canonical projection onto $P$. The module $X'$ is finite dimensional. It has no non-zero preprojective summands because $\pi'$ is surjective and thus $X^\prime$ in $\Gen{L}$. Further, it has no non-zero preinjective summands, because they would have to lie in $\Ker {g'}=P$. But then $X'$ is regular, and $\Ker {g'}$ must be regular as well, a contradiction.

Now assume that $\mathfrak{a}(\Gen{L})$ is closed under coproducts. Then  all direct limits of regular modules, so in particular all Pr\"ufer modules, would belong to  $\mathfrak{a}(\Gen{L})$. But this is not possible: for any Pr\"ufer module $S_\infty$ there is a short exact sequence 
$$0\longrightarrow A\longrightarrow A_S\longrightarrow S_\infty\,^{(2)}\longrightarrow 0$$ 
where $A_S$ is the universal localisation of $A$ at the minimal projective presentation of $S$, and $A_S\in \Gen{L}$ since the category of $A_S$-modules identifies with the subcategory $S^\circ\cap S^{\perp_1}$ of $\ModA$, which cannot contain indecomposable preprojective modules.
For details on  universal localisations and  tilting modules over the Kronecker algebra we refer to \cite{AS}.
\end{example}

On the other hand, if $\Tcal$ is a silting class generated by a minimal silting module, then $\mathfrak a(\Tcal)$ is bireflective. Indeed, the following proposition expresses the abelian category $\mathfrak a(\Gen{T})$ as a subcategory that is also closed under products and coproducts.

\begin{proposition}\cite[Remark 5.7]{AMV2}\label{equals a} 
Let $T$ be a minimal silting module with associated partial silting module $T_1$. Then we have that $$\mathfrak a(\Gen{T})=\Gen{T}\cap T_1^\circ.$$ 
\end{proposition}

To summarise, given a minimal silting module $T$,  we can associate to it a canonical partial silting module and, therefore, a silting ring epimorphism, and it turns out that the corresponding bireflective subcategory is precisely  $\mathfrak a(\Gen{T})$. The following  corollary collects these findings.

\begin{corollary}\label{inj1}
There is a commutative diagram as follows.
$$\xymatrix{{\left\{\begin{array}{c}\text{\Small Equivalence classes}\\ \text{\Small of minimal}\\ \text{\Small silting $A$-modules} \end{array}\right\}}\ar[rr]^{\alpha}\ar[dr]^{\mathfrak{a}\,\,\,} &  & {\left\{\begin{array}{c}\text{\Small Epiclasses of ring} \\ \text{\Small epimorphisms $A\to B$}\\ \text{\Small with $\Tor_1^A(B,B)=0$} \end{array}\right\}}\ar[dl]_{\,\,\,\epsilon}\\ & {\left\{\begin{array}{c}\text{\Small  Bireflective }\\ \text{\Small extension-closed} \\ \text{\Small subcategories}\\ \text{\Small of $\ModA$}\end{array}\right\}} &}$$
The map $\alpha$ assigns to the minimal silting module $T$ the associated silting ring epimorphism. The map 
$\epsilon$ is the bijection from Theorem \ref{epicl}(1) and
$\mathfrak{a}$  is  the map defined in (\ref{alpha}) for the abelian category  $\ModA$.
\end{corollary}

Over a finite dimensional algebra, the triangle above restricts to the level of finite dimensional modules.

\begin{theorem}\cite{MS1}\label{inj} If $A$ is a finite dimensional algebra, there is a commutative diagram of injections as follows.
$$\xymatrix{{\left\{\begin{array}{c}\text{\Small Equivalence classes} \\ \text{\Small of finite dimensional}\\ \text{\Small silting $A$-modules} \end{array}\right\}}\ar[rr]^{\alpha}\ar[dr]^{\mathfrak a} &  & {\left\{\begin{array}{c}\text{\Small Epiclasses of ring } \\ \text{\Small  epimorphisms $A\to B$ }\\ \text{\Small with $\Tor_1^A(B,B)=0$,} \\ \text{\Small $B$ finite dimensional} \end{array}\right\}}\ar[dl]_{\epsilon}\\ & {\left\{\begin{array}{c}\text{\Small  Functorially }\\ \text{\Small finite wide } \\ \text{\Small subcategories }\\ \text{\Small of $\modA$}
\end{array}\right\}} &}$$
Here, the map $\alpha$ is as above, $\epsilon$ is the bijection from Theorem \ref{epicl}(3), and $\mathfrak{a}$  is  the map defined in (\ref{alpha}) for the abelian category $\modA$. Moreover, ring epimorphisms in the image of the assignment $\alpha$ are universal localisations.
\end{theorem}

\begin{proof} 
Observe that every finite dimensional silting module $T$ is minimal. If $T_1$ is the associated partial silting module (with respect to a projective presentation $\sigma_1$), then it follows from the very construction that the associated silting ring epimorphism is the universal localisation of $A$ at $\{\sigma_1\}$. By \cite[Theorem 3.5]{AMV2}, the algebra $A_{\{{\sigma_1}\}}$ is finite dimensional. In particular, $\alpha$ is well-defined. Moreover, by \cite[Lemma 3.8]{MS1}, the functorially finite torsion class $\gen{T}$  gives rise to  the functorially finite wide subcategory $\mathfrak a(\gen{T})=\gen{T}\cap T_1^\circ=\Xcal_{\sigma_1}\cap\modA$. Therefore, the above diagram commutes. Finally, $\mathfrak a$  is injective by \cite[Proposition 3.9]{MS1}.
\end{proof}

It was shown in \cite{Asai} that the maps $\alpha$ and $\mathfrak{a}$ in Theorem \ref{inj} are  surjective only once  we add further assumptions. We will pass to such a situation in the last section of this note.


\section{From abelian subcategories to torsion classes}

The injectivity of the map $\mathfrak{a}$ in Theorem \ref{inj} is proved in \cite{MS1} by providing 
a left inverse map. Our next aim is to find an analogous map  for large modules.
Proceeding as in  \cite{MS1}, we will try to recover a given torsion class $\Tcal$ from $\mathfrak{a}(\Tcal)$ by considering the class of modules admitting filtrations by modules in $\Gen{\mathfrak{a}(\Tcal)}$.

\begin{definition} Given a module $M$ and an ordinal number $\mu$, we call an ascending chain $\mathcal M = (M_\lambda, \lambda \leq \mu)$ of submodules of $M$ a \emph{filtration} of $M$, if $M_0 = 0$, $M_\mu = M$, and $M_\nu=\bigcup _{\lambda < \nu} M_\lambda (=\varinjlim_{\lambda < \nu}M_\lambda)$ for each limit ordinal $\nu\leq\mu$. Moreover, given a class of modules $\mathcal C$, we call $\mathcal M$ an \emph{$\mathcal C$-filtration} of $M$, provided that each of the consecutive factors $M_{\lambda+1}/M_\lambda$ ($\lambda<\mu$) is isomorphic to a module from $\mathcal C$. A module $M$ admitting an  $\mathcal C$-filtration is said to be \emph{$\mathcal C$-filtered}, and the class of all $\mathcal C$-filtered modules is denoted by $\Filt{\Ccal}$.
\end{definition}

\begin{lemma}\label{torsion}
Let $\Ccal$ be a full subcategory of $\ModA$ closed under coproducts and quotients. Then $\Filt{\Ccal}$ is a torsion class in $\ModA$.
\end{lemma}

\begin{proof}
We verify that $\Filt{\Ccal}$ is closed under coproducts, extensions and quotients. The closure under coproducts and extensions is easy. Indeed, since coproducts are exact in $\ModA$ and $\Ccal$ is coproduct-closed, the coproduct of filtrations with factors in $\Ccal$ is still a $\Ccal$-filtration. On the other hand, given a short exact sequence of the form 
$$0\longrightarrow X\longrightarrow Y\longrightarrow Z\longrightarrow 0$$
with $X$ and $Z$ in $\Filt{\Ccal}$, with filtrations given by, respectively, $(X_\lambda,\lambda \leq  \mu)$ and $(Z_\omega,\omega\leq \nu)$ for some ordinals $\mu$ and $\nu$. For each $Z_\omega$, we can choose a subobject $Y_\omega$ of $Y$ such that $Y_\omega/X\cong Z_\omega$.
We claim that $(W_\theta)_{\theta\leq \mu+\nu}$ given by $W_\theta=X_\theta$ whenever $\theta\leq \mu$ and $W_\theta=Y_\theta$ whenever $\mu<\theta\leq \mu + \nu$ is a $\Ccal$-filtration of $Y$. It suffices to observe that $W_{ \mu+1}/W_\mu=W_{ \mu+1}/X=Y_1/X\cong Z_1$ which lies in $\Ccal$.

It remains to show that $\Filt{\Ccal}$ is  closed under quotients. Given a $\Ccal$-filtration $(X_\lambda:\lambda\leq \mu)$ of $X$ and a surjection $f\colon X\longrightarrow Y$,  we can consider the restrictions $f_\lambda\colon X_\lambda\longrightarrow Y$ for each $\lambda\leq \mu$. Clearly, $(\Img{f_\lambda}:\lambda\leq \mu)$ is a filtration of $Y$ since, for any limit ordinal $\nu\le\mu$, we have 
$$\bigcup\limits_{\lambda<\nu}\Img{f_\lambda}=f(\bigcup\limits_{\lambda<\nu}X_\lambda)=f(X_\nu)=\Img{f_\nu}.$$
We show that indeed $(\Img{f_\lambda}:\lambda\leq \mu)$ is a $\Ccal$-filtration of $Y$. In fact, for each $\lambda<\mu$, we have the following diagram.
$$\xymatrix{0\ar[r]&X_\lambda\ar[r]\ar[d]^{\tilde{f}_\lambda}&  X_{\lambda +1}\ar[r]\ar[d]^{\tilde{f}_\lambda}& X_{\lambda+1}/X_\lambda\ar[r]\ar[d]^g&0\\0\ar[r]&\Img{f_\lambda}\ar[r]& \Img{f_{\lambda+1}}\ar[r]& \Img{f_{\lambda+1}}/\Img{f_\lambda}\ar[r]&0}$$
Since all vertical maps are surjective and $\Ccal$ is closed under quotients, it follows that $ \Img{f_{\lambda+1}}/\Img{f_\lambda}$ lies in $\Ccal$, as wanted.
\end{proof}

If $\Ccal$ is a class of modules, we use the abbreviation $\FiltGen{\Ccal}=\Filt{\Gen{\Ccal}}$. By the lemma above, this  is the smallest torsion class in $\ModA$ containing $\Ccal$.

\begin{lemma}\label{sub}
Let $\Ccal$ be a full subcategory of $\ModA$ closed under kernels, cokernels, coproducts and extensions. Then $\Ccal$ is closed under subobjects in $\FiltGen{\Ccal}$.
\end{lemma}

\begin{proof}
Let $C$ be in $\Ccal$ and $f:Y\longrightarrow C$ be an injection with $Y$ in $\FiltGen{\Ccal}$. 
Suppose that $Y$ lies in $\Gen{\Ccal}$. Since $\Ccal$ is closed under coproducts, there is a surjection $g:W\longrightarrow Y$ for some $W$ in $\Ccal$. Hence, $Y$ is isomorphic to the image of the map $fg\colon W\longrightarrow C$. Since $\Ccal$ is closed under kernels and cokernels, $Y$ lies in $\Ccal$. 

Now we consider a $\Gen{\Ccal}$-filtration $(Y_\lambda:\lambda\leq \mu)$ of $Y$ and we proceed by transfinite induction. 
Let $\lambda< \mu$ and suppose $Y_\lambda$ lies in $\Ccal$.  Since $\Ccal$ is closed under cokernels,  $C/Y_\lambda$ lies in $\Ccal$ as well (we can assume without loss of generality that $f$ is an inclusion map). Hence $Y_{\lambda+1}/Y_\lambda$, which lies in $\Gen{\Ccal}$ and is a subobject of $C/Y_\lambda$,  also lies in $\Ccal$. Since $\Ccal$ is extension-closed, we conclude that $Y_{\lambda+1}$ lies in $\Ccal$.
Finally, let $\nu\leq\mu$ be a limit ordinal, and assume that $Y_\lambda$  lies in $\Ccal$ for any $\lambda<\nu$. Since $\Ccal$ is closed under coproducts and cokernels, it is also closed under direct limits and, therefore, $Y_\nu=\varinjlim_{\lambda<\nu} Y_\lambda$ lies in $\Ccal$.
\end{proof}

The following result can be viewed as a large version of \cite[Proposition 3.3]{MS1}.

\begin{proposition}\label{afiltgen}
Let $\Ccal$ be a full subcategory of $\ModA$ closed under kernels, cokernels, coproducts and extensions. Then we have $\mathfrak a(\FiltGen{\Ccal})=\Ccal$.
\end{proposition}

\begin{proof}
Let us first show that $\Ccal$ is contained in $\mathfrak a(\FiltGen{\Ccal})$. Consider a map $f:X\longrightarrow C$ with $X$ in $\FiltGen{\Ccal}$ and $C$ in $\Ccal$. By the previous lemma, without loss of generality we may assume $f$ to be surjective. We have to show that $K:=\Ker{f}$ lies in $\FiltGen{\Ccal}$. If $X$ lies in $\Gen{\Ccal}$, then there is $W$ in $\Ccal$ and a surjection $g:W\longrightarrow X$ inducing the following diagram.
$$\xymatrix{0\ar[r]&L\ar[r]^a\ar[d]^{b}&  W\ar[r]^{fg}\ar[d]^g& C\ar[r]\ar@{=}[d]&0\\0\ar[r]& K\ar[r]& X\ar[r]^f& C\ar[r]&0}$$
Since $g$ is a surjection, so is $b$, and since $\Ccal$ is closed under kernels, $L$ lies in $\Ccal$. Hence, $K$ lies in $\Gen{\Ccal}$. Now consider a $\Gen{\Ccal}$-filtration $(X_\lambda:\lambda\leq \mu)$ of $X$. It is easy to see that $(K_\lambda:\lambda\leq \mu)$, where $K_\lambda$ is the kernel of the restriction $f_\lambda:X_\lambda\longrightarrow C$ of $f$, is a filtration of $K=\Ker{f}$.
We show that this is in fact a $\Gen{\Ccal}$-filtration. Let $\lambda<\mu$ and consider the following commutative diagram
$$\xymatrix{0\ar[r]&X_\lambda\ar[r]^a\ar[d]^{\tilde{f}_\lambda}&  X_{\lambda+1}\ar[r]^{p}\ar[d]^{f_{\lambda+1}}& X_{\lambda+1}/X_\lambda\ar[r]\ar[d]^b&0\\0\ar[r]& Im(f_\lambda)\ar[r]& C\ar[r]^{d \qquad}& C/Im(f_\lambda)\ar[r]&0}$$
with $\tilde{f}_\lambda$ being the natural factorisation of $f_\lambda$ via its image.
By Lemma~\ref{sub}, we have that $\Img{f_\lambda}$ lies in $\Ccal$. Therefore, so does $C/\Img{f_\lambda}$. By the previous argument, using that  $X_{\lambda+1}/X_\lambda$ lies in $\Gen{\Ccal}$, we get that $\Ker{b}\cong K_{\lambda+1}/K_\lambda$ lies in $\Gen{\Ccal}$.

It remains to show  that $\mathfrak a(\FiltGen{\Ccal})\subseteq \Ccal$. Take $Y$  in $\mathfrak a(\FiltGen{\Ccal})$. Then $Y$ lies in $\FiltGen{\Ccal}$, and we proceed by transfinite induction as before. If $Y$ lies in $\Gen{\Ccal}$, then there is a surjection $f:C\longrightarrow Y$ with $C$ in $\Ccal$ (as $\Ccal$ is closed under coproducts). Since $Y$ lies in $\mathfrak a(\FiltGen{\Ccal})$, the kernel $K:=\Ker{f}$ lies in $\FiltGen{\Ccal}$, and even in $\Ccal$ by Lemma~\ref{sub}. Then $Y$  lies in $\Ccal$ as $\Ccal$ is closed under cokernels. Now suppose that $Y$ has a $\Gen{\Ccal}$-filtration $(Y_\lambda:\lambda\leq \mu)$. Let $\lambda< \mu$ and suppose that $Y_\lambda$ lies in $\Ccal$. Then $\mathfrak a(\FiltGen{\Ccal})$, which contains $\Ccal$ and is clearly a subcategory closed under subobjects in $\FiltGen{\Ccal}$, must contain both $Y_\lambda$   and $Y_{\lambda+1}$, and therefore also  $Y_{\lambda+1}/Y_\lambda$ by Proposition~\ref{abelian}. But $Y_{\lambda+1}/Y_\lambda$ lies in $\Gen{\Ccal}$ and, hence, it lies in $\Ccal$ by the argument above. Since $\Ccal$ is closed under extensions, we infer that also $Y_{\lambda+1}$ belongs to $\Ccal$.
Finally, let  $\nu\leq\mu$ be a limit ordinal and assume that all $Y_\lambda$ with $\lambda<\nu$ belong to $\Ccal$. Since
$\Ccal$ is  closed under coproducts and cokernels, it follows that $Y_\nu=\varinjlim_{\lambda<\mu}Y_\lambda$ also lies in $\Ccal$, thus finishing the proof.
\end{proof}

We remark that under our assumptions above  $\Ccal$ is not necessarily bireflective.

\begin{example}  
Let $A$ be the Kronecker algebra, and let $\tube$ denote the class of all finite dimensional indecomposable regular modules.
Take $\Ccal$ to be the direct limit closure of $\add\tube$. This is a full subcategory of $\ModA$ closed under kernels, cokernels, coproducts and extensions which contains all Pr\"ufer modules and lies in the torsion class $\Gen{\tube}=\Gen{\Ccal}=\FiltGen{\Ccal}$. The corresponding torsionfree class $\tube^o$ is the class of all modules cogenerated by the generic module $G$. For details we refer to \cite{RR}. Notice that $G$  is a summand of any countable product of copies of a Pr\"ufer module. This shows that $\Ccal$ it is not closed under products (so it is not bireflective).
\end{example}

We now aim at a large version of \cite[Proposition 3.9]{MS1}. Let us first recall the following result.
\begin{lemma}\cite[(4.4) Lemma]{CB}\label{cb} Let $A$ be a right coherent ring, and let $(\Xcal,\Ycal)$ be a torsion pair in $\modA$. Then the classes $\Tcal:=\varinjlim \Xcal$ and $\Fcal:=\varinjlim \Ycal$ consisting of all direct limits of modules in $\Xcal$ and $\Ycal$, respectively, form a torsion pair $(\Tcal, \Fcal)$ in $\ModA$ and, furthermore, $\Tcal=\Gen{\Xcal}$ and $\Fcal=\Xcal ^o$. 
\end{lemma} 

\begin{proposition}\label{filtgena}
Let $A$ be a finite dimensional algebra and let $T$ be a finite dimensional silting module. Then $\FiltGen{\mathfrak a(\Gen{T})}=\Gen{T}$.
\end{proposition}
\begin{proof}
By construction we have $\FiltGen{\mathfrak a(\Gen{T})}\subseteq\Gen{T}$. For the 
reverse inclusion, we observe that $\Gen{T}\cap\modA$ coincides with the class $\gen{T}$ of all finitely $T$-generated modules, which is a functorially finite torsion class in $\modA$ by \cite[Theorem 2.7]{AIR}.   It follows from Lemma~\ref{cb} that $\Gen{T}=\varinjlim(\gen{T})$, and by \cite[Proposition 3.9]{MS1} we know that $\gen{T}$ consists of all modules with a finite filtration by quotients of modules in $\mathfrak a(\gen{T})$ (here   $\mathfrak a$  is  the map defined in (\ref{alpha}) for the abelian category  $\modA$). If we prove that $\mathfrak a (\gen{T})$ is contained in $\mathfrak a(\Gen{T})$ (here $\mathfrak a$  is  now  defined  for the abelian category  $\ModA$), then $\gen{T}$ will be contained in $\FiltGen{\mathfrak a(\Gen{T})}$ and, hence, so will $\Gen{T}=\varinjlim \gen{T}$, concluding our proof.
So, let us show that $\mathfrak a (\gen{T})\subseteq \mathfrak a(\Gen{T})$. Take a module $X$ in $\mathfrak a(\gen{T})$. Of course $X$ lies also in $\Gen{T}$.  Consider a map $g:Y\ra X$ with $Y$ in $\Gen{T}$. Then there is a direct system $(Y_i)_{i\in I}$ in $\gen{T}$ such that $Y=\varinjlim Y_i$, and the induced maps $g_i:Y_i\ra X$ form a direct system with $\Ker{g}=\varinjlim \Ker{g_i}$. Since by assumption $\Ker{g_i}$ belongs to $\gen{T}$ for any $i$, we conclude that $\Ker{g}$ lies in $\Gen{T}$.
\end{proof}

We obtain a large version of Theorem \ref{inj} together with an explicit description of the image of  $\mathfrak a$.

\begin{corollary}\label{cor inject}
If $A$ is a finite dimensional algebra, there is a commutative diagram of injections as follows.
$$
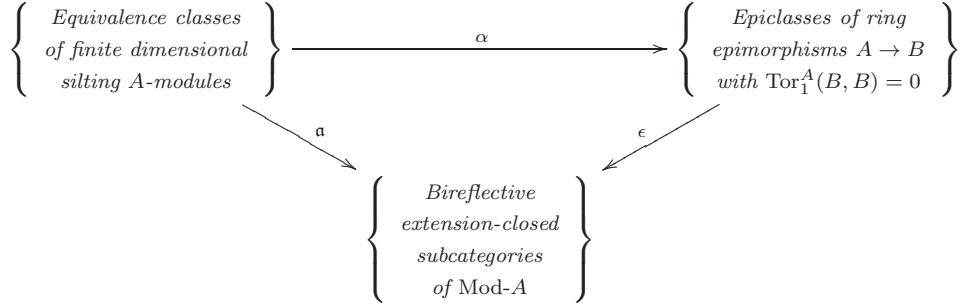
\xymatrix{{\left\{\begin{array}{c}\text{\Small Equivalence classes} \\ \text{\Small of finite dimensional}\\ \text{\Small silting $A$-modules} \end{array}\right\}}\ar[rr]^{\alpha}\ar[dr]^{\mathfrak a} &  & {\left\{\begin{array}{c}\text{\Small Epiclasses of ring} \\ \text{\Small  epimorphisms $A\to B$}\\ \text{\Small with $\Tor_1^A(B,B)=0$} \end{array}\right\}}\ar[dl]_{\epsilon}\\ & {\left\{\begin{array}{c}\text{\Small  Bireflective}\\ \text{\Small extension-closed} \\ \text{\Small subcategories}\\ \text{\Small of $\ModA$}
\end{array}\right\}} &}$$
The map $\alpha$ assigns to the minimal silting module $T$ the associated silting ring epimorphism. The map $\epsilon$ is the bijection from Theorem \ref{epicl}(1) and $\mathfrak{a}$  is  the map defined in (\ref{alpha}) for the abelian category  $\ModA$. Moreover, the image of   $\mathfrak a$ consists of the bireflective extension-closed  subcategories $\Ccal$ in $\ModA$ for which $\FiltGen{\Ccal}$ is generated by a finite dimensional silting module.
\end{corollary}

\begin{proof}
By Proposition \ref{filtgena}, the map $\mathfrak a$ yields an injection from equivalence classes of finite dimensional silting modules  to bireflective extension-closed subcategories of $\ModA$. The description of the image follows from Proposition \ref{afiltgen}.
\end{proof}

Note that, if all torsion classes in $\ModA$ are generated by a finite dimensional silting module, then $\mathfrak a$ is a bijection. Let us turn to algebras with such property.


\section{$\tau$-tilting finite algebras}

In this section, $A$ will be a finite dimensional algebra and we denote by $\tau$ the Auslander-Reiten translate. Recall the following definitions.

\begin{definition}\cite{AIR} A finite dimensional $A$-module $M$ is said to be $\tau$-{\it rigid} if  $\Hom_A(M,\tau M)=0$, and it is called $\tau$-{\it tilting} if, in addition,  the number of non-isomorphic indecomposable direct summands of $M$ coincides with the number of isomorphism classes of simple $A$-modules. Finally, $M$ is said to be {\it support $\tau$-tilting}, if it is $\tau$-tilting over $A/AeA$ for an idempotent $e\in A$.
\end{definition}

The next theorem provides a characterisation of algebras that have only finitely many $\tau$-tilting modules, up to additive equivalence. Recall that a finite dimensional $A$-module is said to be a \textit{brick} if its endomorphism ring is a division ring.

\begin{theorem}\cite{DIJ}\label{dij} The following statements are equivalent for a finite dimensional algebra $A$.
\begin{enumerate}
\item There are only finitely many isomorphism classes of indecomposable $\tau$-rigid $A$-modules.
\item There are only finitely many isomorphism classes of basic $\tau$-tilting $A$-modules.
\item There are only finitely many isomorphism classes of bricks in $\modA$.
\item There are only finitely many torsion classes in $\modA$.
\item Every torsion class in $\modA$ is functorially finite.
\end{enumerate}
\end{theorem}

\begin{definition} 
A finite dimensional algebra $A$ is called $\tau$-{\it tilting finite} if it satisfies the equivalent conditions of Theorem \ref{dij}.
\end{definition}

\begin{example}
Let $A$ be a preprojective algebra of Dynkin type. Then $A$ is $\tau$-tilting finite by \cite{Mi}. Indeed, the isomorphism classes of basic support $\tau$-tilting $A$-modules correspond bijectively to the elements of the Weyl group associated to the Dynkin graph. It is also not hard to check that every brick in $\modA$ has no non-trivial self-extensions  and, hence, its dimension vector needs to be a positive root of the underlying Dynkin graph. In particular, up to isomorphism, there are only finitely many bricks in $\modA$. 
\end{example}

\begin{example}
The following example is due to \cite{R}. Let $Q$ be the quiver
$$
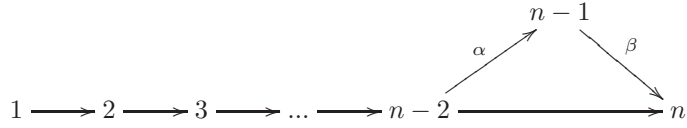
\xymatrix{ & &  &  & & n-1\ar[dr]^\beta & \\ 1\ar[r] & 2\ar[r] & 3\ar[r] & ...\ar[r] & n-2\ar[rr]\ar[ur]^\alpha & & n}$$
and let $A$ be the finite dimensional algebra given by $Q$ modulo the relation $\beta\alpha$. Then $A$ has global dimension two. It was observed in \cite{R} that for $n\geq 9$ the algebra $A$ is of wild representation type and that, independently of $n$, there are only finitely many bricks in $\modA$, up to isomorphism. In particular, $A$ is $\tau$-tilting finite. Note that, by choosing $n$ large enough and adding relations to the ideal above, we can construct further $\tau$-tilting finite algebras of wild representation type and of any finite global dimension.
\end{example}

Since support $\tau$-tilting modules over $A$ are precisely the finite dimensional silting modules \cite[Proposition 3.15]{AMV1}, it is only natural to ask whether any torsion class in $\ModA$ whose intersection with $\modA$ is functorially finite is a silting torsion class. The following lemma answers this question positively. For its proof, recall that a subcategory of $\ModA$ is said to be \textit{definable} if it is closed under products, direct limits and pure submodules ($X$ is a \textit{pure submodule} of $Y$ if the inclusion of $X$ into $Y$ remains injective after tensoring it with any left $A$-module).

\begin{lemma}\label{ff}
Let $\Tcal$ be a torsion class in $\ModA$ such that $\Tcal\cap\modA$ is functorially finite. Then there is a finite dimensional silting module $T$ such that $\Tcal=\Gen{T}$.
\end{lemma}

\begin{proof}
If  $\Tcal\cap\modA$ is a functorially finite torsion class in $\modA$, it has the form $\gen{T}$ for a finite dimensional silting module $T$, by \cite[Theorem 2.7]{AIR}. Then the direct limit closure $\varinjlim\gen{T}$ coincides with $\Gen{T}$ by Lemma~\ref{cb}, and it is obviously contained in $\Tcal$.
Conversely, since every module $X$ is a pure submodule of the direct product of all its finitely generated quotients, see \cite[2.2, Example 3]{CB}, any module $X$ in $\Tcal$ lies in the smallest definable subcategory containing $\Tcal\cap\modA=\gen{T}$. Since silting classes are definable by \cite[Corollary 3.5]{AMV1}, it is then clear that the smallest definable subcategory containing $\gen{T}$ is $\Gen{T}$, thus proving that $\Tcal=\Gen{T}$.
\end{proof}

We are now ready for the main results of this note. Observe that the diagram of injections from Theorem \ref{inj} becomes a diagram of bijections if we assume our algebra $A$ to be $\tau$-tilting finite, see \cite[Corollary 3.11]{MS1}. In particular, it follows that an algebra $A$ is $\tau$-tilting finite if and only if there are only finitely many epiclasses of ring epimorphisms $A\ra B$ with $\Tor^A_1(B,B)=0$ and such that $B$ is again a finite dimensional algebra. The following result shows that an analogous statement holds true if we allow for large modules and ring epimorphisms $A\ra B$ for which $B$ is no longer assumed to be a finite dimensional algebra.
 
\begin{theorem}\label{bij}
If  $A$ is a $\tau$-{tilting finite} algebra, there is a commutative diagram of bijections as follows.
$$\xymatrix{{\left\{\begin{array}{c}\text{\Small Equivalence classes} \\ \text{\Small of finite dimensional}\\ \text{\Small silting $A$-modules} \end{array}\right\}}\ar[rr]^{\alpha}\ar[dr]^{\mathfrak a} &  & {\left\{\begin{array}{c}\text{\Small Epiclasses of ring} \\ \text{\Small  epimorphisms $A\to B$}\\ \text{\Small with $\Tor_1^A(B,B)=0$} \end{array}\right\}}\ar[dl]_{\epsilon}\\ & {\left\{\begin{array}{c}\text{\Small  Bireflective}\\ \text{\Small extension-closed} \\ \text{\Small subcategories}\\ \text{\Small of $\ModA$}
\end{array}\right\}} &}$$
The maps are defined as in Corollary \ref{cor inject}.
Moreover, it follows that all ring epimorphisms $A\ra B$ with $\Tor_1^A(B,B)=0$ are universal localisations with $dim\, B<\infty$.
\end{theorem}

\begin{proof} 
By Theorem \ref{dij} and Lemma \ref{ff}, every torsion class in $\ModA$ is generated by a finite dimensional silting $A$-module. The diagram of bijections then follows from Corollary \ref{cor inject}.
For the last statement we refer back to Theorem \ref{inj}.
\end{proof}

\begin{theorem}\label{further eq conditions}
The following statements are equivalent for a finite dimensional algebra $A$.
\begin{enumerate}
\item $A$ is $\tau$-tilting finite.
\item Every torsion class $\Tcal$ in $\ModA$ has the form $\Tcal=\Gen{T}$ for a finite dimensional silting module~$T$.
\item Every silting $A$-module is finite dimensional up to equivalence.
\item There are only finitely many epiclasses of ring epimorphisms  $A\longrightarrow B$ with $\Tor_1^A(B,B)=0$.
\end{enumerate}
\end{theorem}

\begin{proof}
(1) $\Rightarrow$ (2) is Lemma~\ref{ff}, and (2) $\Rightarrow$ (3) is clear. Moreover,  (1) $\Rightarrow$ (4) is a direct consequence of Theorem~\ref{bij} and (4) $\Rightarrow$ (1) follows from Corollary \ref{cor inject}.

(3) $\Rightarrow$ (1):
If $A$ is not $\tau$-{tilting finite}, there must be a torsion class in $\modA$ that is not functorially finite. As argued in the proof of \cite[Theorem 3.8]{DIJ}, it follows that there is an infinite sequence of finite dimensional silting modules $T_1,T_2,\ldots$ inducing  a strictly decreasing chain of torsion classes  
$$\modA\supset\gen{T_1}\supset \gen{T_2}\supset\ldots$$
Passing to $\ModA$, we note that the equality $\Gen{T_i}=\Gen{T_{i+1}}$ would imply  $\add{T_i}=\add{T_{i+1}}$.   Hence, there is also a strictly decreasing chain of torsion classes in $\ModA$ $$\ModA\supset\Gen{T_1}\supset \Gen{T_2}\supset\ldots$$
and its intersection gives rise to a definable torsion class $\bigcap\Gen{T_i}$ which is a silting class as a consequence of \cite[Corollary 3.8]{AH}. By assumption, there is a finite dimensional silting module $T$ such that $\bigcap\Gen{T_i}=\Gen{T}$. It follows  that $\bigcap\gen{T_i}=\Gen{T}\cap\modA=\gen{T}$, contradicting \cite[Lemma 3.10 (b)]{DIJ}.
\end{proof}

Following Theorem \ref{further eq conditions}, we provide a further class of examples of $\tau$-tilting finite algebras. This class was discussed in \cite{GdP}. 

\begin{example}
Suppose that the algebra $A$ fulfills the following two conditions:
\begin{itemize}
\item there are elements $c_1$ and $c_2$ in the center of $A$ such that the algebra $A/c_iA$ is representation-finite for $i=1,2$;
\item the sum of the annihilators of $c_1$ and $c_2$ is $A$.
\end{itemize}
As argued in \cite[Subsection 2.4]{GdP}, then there are only finitely many epiclasses of ring epimorphisms $A\ra B$. In particular, by Theorem \ref{further eq conditions}, $A$ is $\tau$-tilting finite. An explicit example of a representation-infinite algebra fulfilling our assumptions, as found in \cite[Subsection 2.4]{GdP}, is the quotient of the path algebra of 
$$\xymatrix{\bullet\ar@(dl,ul)^{\alpha}\ar[r]^\gamma&\bullet&\bullet\ar[l]_\delta\ar@(dr,ur)_\beta}$$
by the ideal generated by the relations $\alpha^2=\gamma\alpha=\beta^3=\delta\beta^2=0$. The elements $c_1$ and $c_2$ above can be chosen to be, respectively, the elements $\alpha$ and $\beta^2$.
\end{example}

\begin{remark} If $A$ is $\tau$-{tilting finite}, it follows from Theorem \ref{bij} that whenever $A\ra B$ is a ring epimorphism fulfilling $\Tor_1^A(B,B)=0$, then also the algebra $B$ is finite dimensional (and $\tau$-tilting finite). We do not know about the converse. Is it true that $A$ does not admit an infinite dimensional ring epimorphism $A\ra B$ with $\Tor_1^A(B,B)=0$ if and only if it is $\tau$-{tilting finite}?
\end{remark}


\bibliographystyle{amsalpha}

\end{document}